\documentclass[12pt,leqno]{amsart}

\usepackage{amsmath,amsthm,amscd,amssymb,eucal,graphicx}
\swapnumbers

\newtheorem{thm}[subsection]{Theorem}
\newtheorem*{thm*}{Theorem}
\newtheorem{cor}[subsection]{Corollary}
\newtheorem{lem}[subsection]{Lemma}
\newtheorem*{lem*}{Lemma}
\newtheorem{prop}[subsection]{Proposition}
\newtheorem*{prop*}{Proposition}

\theoremstyle{definition}

\newtheorem{example}[subsection]{Example}

\numberwithin{equation}{subsection}

\newcommand{\quash}[1]{}

\newcommand{\CC}{\mathbb C}
\newcommand{\PP}{\mathbb P}
\newcommand{\RR}{\mathbb R}

\newcommand{\ZZ}{\mathbb Z}

\DeclareMathOperator{\Pf}{Pf}
\DeclareMathOperator{\Eu}{\rm Eu}
\DeclareMathOperator{\CM}{\rm CM}
\DeclareMathOperator{\Gr}{\rm Gr}
\DeclareMathOperator{\End}{\rm End}

\newcommand{\oX}{{\overline{X}}}
\newcommand{\ocaX}{{\overline{\mathcal{X}}}}
\newcommand{\one}{1\hskip-3.5pt1}

\begin{document}

\title{Pfaffian integrals and invariants of singular varieties}
\author{Paolo Aluffi}\thanks{Department of Mathematics, Florida State
  University, Tallahassee, FL}
\author{Mark Goresky}\thanks{School of Mathematics, Institute for
  Advanced Study, Princeton NJ} 

\dedicatory
{Dedicated to L\^e D\~ung Tr\'ang on the occasion of his 70th
birthday}

\maketitle

\section{Introduction}\subsection{}
Chern classes of singular varieties have enjoyed an enormous
development over the last forty years with hundreds of published
papers and a number of books dedicated to the subject.
The simple nature of the Gauss Bonnet theorem has received
less attention, after a short announcement in \cite{Dubson}; see also 
\cite{Satake}, \cite{Siersma}, \cite{Fu}. In~\cite{Buryak1, Buryak2},
Alexandr Buryak studies the more general case of characteristic numbers
for singular varieties, showing that any collection of characteristic numbers 
can be realized by a singular complex projective variety.
In this paper we reformulate some of the 
statements in \cite{Dubson, Buryak2} and give a number of extensions, 
applications and examples as we now describe.

Let $\oX$ be an $n$ dimensional complex projective algebraic
variety and let $X$ be the nonsingular part.  Let $\Omega^r(X;\CC)$
denote the smooth differential $r$-forms on $X$.  Endow $X$ with the
K\"ahler metric that is induced from the Fubini Study metric on the
ambient projective space.  Associated to this metric there is a unique
torsion-free (Levi Civita) connection $\nabla_X$ and its curvature
$2$-form $\Theta\in \Omega^2(X; \End(TX))$ with values in the
endomorphism bundle of the tangent bundle $TX$.  Let
\[ 
\Pf(\Theta_{\RR})  \in\Omega^{2n}(X;\RR) 
\]
be the Pfaffian $2n$-form of the underlying real vector bundle with
connection.  If $X= \oX$ is compact and nonsingular, then the Gauss
Bonnet theorem states that the integral
\[
(2\pi)^{-n}\int_X \Pf(\Theta_{\RR}) = \chi(X)
\] 
is equal to the Euler characteristic of $X$.  In the noncompact case,
however, this is far from true.  It does not even seem obvious, a
priori, that the above integral is finite, or that it is independent
of the embedding in projective space.

\subsection{}  
In \cite{Mac}, R. MacPherson defines, for each point $y \in \oX$, the
local Euler obstruction $\Eu_y(\oX)\in \ZZ$.  It is constant as the
point $y$ varies within any stratum of a Whitney stratification of $X$
and has become a well studied local invariant of singular spaces.

\begin{thm}\label{thm-integral}
Choose a Whitney stratification of $\oX$ so that $X$ is the largest
stratum.  Then the Gauss Bonnet integral is equal to the following
sum,
\begin{equation}\label{eqn-sum}
(2 \pi)^{-n} \int_X \Pf(\Theta_{\RR}) = \chi(X) + \sum_{Y<X} \Eu_y(\oX) \chi(Y)
\end{equation}
(where the sum is over singular strata $Y$ and where $y\in Y$).  It is
an integer, is independent of the embedding, and is constant within
any (Whitney) equisingular (\cite{Teissier}) family.
\end{thm}
This result is extended to higher Chern classes in Proposition
\ref{prop-main} and proven in~\S\ref{sec-proof}. Our argument
relies on the Nash blow-up (the same strategy is also used in~\cite{Buryak2}).
In~\S\ref{sec-intrinsic} we give an intrinsic interpretation of the integral,
relying on {\em controlled differential forms\/} rather than Nash blow-up.
The more general case of
characteristic numbers is treated in~\S\ref{sec:MCnos}, and extended to
degrees of products of Mather Chern classes. We prove that the information
carried by the degrees of Mather Chern classes is equivalent to a suitable
collection of Pfaffian integrals. In \S \ref{sec-families} we answer a query 
of Robert Langlands concerning the behavior of the Gauss Bonnet integral 
in a flat family. In~\S\ref{sec:Efq} we briefly mention natural extensions.

\subsection{}  
We are grateful to Robert Langlands for asking these questions and for
prompting us to publish these results.
We wish to thank J\"org Sch\"urmann for reviewing an earlier version of this 
note and for informing us of the work of Alexandr Buryak (\cite{Buryak1, 
Buryak2}).

We are pleased to dedicate this paper to L\^e D\~ung Tr\'ang in recognition 
of his tremendous influence on singularity theory.

\section{Higher Chern classes}

\subsection{}  
Suppose, as in the introduction, that $X$ is the nonsingular part of a
complex projective variety $\oX$ equipped with the Fubini Study metric
and associated Levi Civita connection $\nabla$ and curvature form
$\Theta$.  The {\em total Chern form}
\begin{equation}\label{eqn-chern}
 c^*(\Theta) = \det \left( I + \frac{\sqrt{-1}}{2\pi}\Theta\right) 
\in \Omega^*(X;\CC)
\end{equation}
is a sum of homogeneous terms of even degree, and the degree $2r$
part, denoted $c^r(\Theta)$, is equal to $(\sqrt{-1}/2\pi)^r$ times
the $r$-th elementary symmetric function of the eigenvalues of
$\Theta$.

\subsection{}
Let $b:\widehat{X} \to \oX$ be the Nash blowup of $\overline X$; it is
defined to be the closure of the image of $X \to \Gr_n(T \PP^n)$ in
the Grassmann bundle of $n$ dimensional subspaces of the tangent
bundle of $\PP^n$, under the Gauss map $x \mapsto T_xX$.  The
tautological $n$-plane bundle on $\Gr_n(T \PP^n)$ when restricted to
$\widehat{X}$ is a vector bundle, denoted $\xi$, that extends the
tangent bundle $TX \to X$. Let $c^r(\xi) \in H^{2r}(\widehat{X};\ZZ)$
denote its $r$-th Chern class.  Recall \cite{Mac} that the {\em Mather
  Chern class} of $\oX$ is the homology class $c^M(\overline X)$ whose
component of dimension~$n-r$ is
\[
c^M_{n-r}(\oX) = b_*\left(c^r(\xi) \cap [\widehat{X}]\right) 
\in H_{2n-2r}(\oX;\ZZ)
\]   
where $[\widehat{X}] \in H_{2n}(\widehat{X};\ZZ)$ denotes the
fundamental (orientation) class of $\widehat{X}$.

\subsection{} 
Each point $y \in \widehat{X}$ corresponds to an $n$-dimensional
subspace $\xi_y$ of $T_{\pi(y)}\PP^n$ and so it inherits a Hermitian
metric from the Fubini Study metric on $\PP^n$.  Thus, not only does
the tangent bundle of $X$ extend to a vector bundle on $\widehat{X}$
but also the K\"ahler metric on $X$ extends to a Hermitian metric on
the budle $\xi$. We would like to say that the canonical connection
$\nabla$ and its curvature 2-form $\Theta$ also extends to
$\widehat{X}$ but unfortunately this variety may still be singular
(although it may be possible, with some work, to make sense of these
notions in this setting).  We avoid these technical difficulties by
passing to a resolution of singularities
\begin{equation*}
\begin{CD}
\widetilde{X} @>{\pi}>> \widehat{X} @>{b}>>\oX
\end{CD}
\end{equation*}
of the Nash blowup $\widehat{X}$.

\subsection{}
The Hermitian metric on $\xi$ pulls back to a Hermitian metric on the
vector bundle $\pi^*(\xi)$.  Let $\widetilde{\nabla}$ be the
associated connection; it therefore extends the Levi Civita connection
$\nabla$ on $TX$.  Let $\widetilde{\Theta}$ denote its curvature
2-form (with values in $\End(\pi^*(\xi))$) and let
\[
c^*(\widetilde\Theta) = \det\left(I + \frac{\sqrt{-1}}{2\pi}
\widetilde{\Theta}\right)
\] 
be the corresponding total Chern form on $\widetilde{X}$ with its
associated cohomology class $[ c^*(\widetilde{\Theta})] \in
H^*(\widetilde{X};\CC).$ Let $\deg :H_0(\oX;\ZZ) \to \ZZ$ denote the
augmentation.

\begin{prop}\label{prop-main}
The homology class
\begin{equation}\label{eqn-main}
b_* \pi_*\left([c^*(\widetilde{\Theta})] \cap [\widetilde{X}]\right) 
= c^M(\oX) 
\end{equation}
coincides with the (total) Mather Chern class of $\oX$.  Consequently
the Gauss Bonnet integral
\begin{equation}\label{eqn-thm}
(2\pi)^{-n}\int_X \Pf(\Theta_{\RR}) = \deg c^M_0(\oX)
\end{equation} 
is equal to the degree zero part of the Mather-Chern class. 
\end{prop}

\begin{proof}
The usual Chern Weil theorem describes the Chern class of $\pi^*(\xi)$
as follows:
\[ 
c^*(\pi^*(\xi)) = [c^*(\widetilde{\Theta})].
\]
Therefore
\begin{align*}
b_*\pi_*([c^*(\widetilde\Theta)] \cap [\widetilde{X}]) 
&= b_*\pi_*(\pi^*c^*(\xi)\cap [\widetilde{X}])\\
&= b_*(c^*(\xi) \cap [\widehat{X}])\\
&= c^M(\oX) 
\end{align*}
which proves equation (\ref{eqn-main}).  For a complex rank $n$ vector
bundle (such as $\pi^*(\xi)$) the top Chern form agrees with the
Pfaffian (see, for example, \cite{Ballman} p.~186), that is,
\[ 
c^n(\widetilde\Theta) = \det\left(\frac{\sqrt{-1}}{2\pi}
\widetilde\Theta\right) = 
\Pf\left(\frac{1}{2\pi}\widetilde\Theta_{\RR}\right).
\] 
Moreover, the differential form $\Pf(\widetilde{\Theta}_{\RR})$ on
$\widetilde{X}$ is smooth.  It agrees with $\Pf(\Theta_{\RR})$ on $X$,
and $\widetilde{X} - X$ has measure zero.  So the degree zero part of
equation (\ref{eqn-main}) becomes
\[ 
\int_{\widetilde{X}} c^n(\widetilde\Theta) 
= (2\pi)^{-n}\int_{\widetilde{X}}\Pf(\widetilde\Theta_{\RR})
= (2\pi)^{-n}\int_X \Pf(\Theta_{\RR}).\qedhere
\]
\end{proof}

\section{Proof of Theorem \ref{thm-integral}}\label{sec-proof}
\subsection{}\label{ss:Eu}
As in \cite{Mac}, the Euler obstruction determines an isomorphism $T$
from the group of algebraic cycles on $\oX$ to the group of
constructible functions on $\oX$ by
\[ 
T(\sum_ia_iV_i)(p) = \sum_i a_i \Eu_p(V_i).
\]
The (MacPherson) total Chern class of a constructible function $F$ is
defined to be $c_*(F) = c^MT^{-1}(F)$.  It follows that the Mather
Chern class of $\oX$ is the MacPherson Chern class of the
constructible function that is $F(p) = \Eu_p(\oX)$.  Proposition
\ref{prop-main} says that the Gauss Bonnet integral is equal to the
degree zero part of this class which is therefore the Euler
characteristic of this constructible function $F$, giving equation
\eqref{eqn-sum}.

\subsection{}
The Euler obstruction is constant on strata of a Whitney stratification 
(\cite{BrSch}) and it is independent of the embedding so the right
side of (\ref{eqn-sum}) is independent of the embedding and is
constant within any Whitney-equisingular family.  (In fact, the Nash
blowup is independent of the embedding and can be defined
intrinsically on $\oX$ according to \cite{No}.)  This completes the
proof of Theorem \ref{thm-integral}.  \qed

\section{The map to homology}\label{sec-intrinsic}
\subsection{}
It is possible to give an intrinsic description of the manner in which
the differential form $c^*(\Theta)\in \Omega^*(X)$ gives rise to a
homology class in $H_*(\oX;\CC)$ in equation \eqref{eqn-main}, without
referring to the Nash blowup.  We give an outline only because the
result uses various technicalities involving Whitney stratifications,
and it is not essential for the development in this paper.

\subsection{}
For any Whitney stratification of $\oX$ and for an associated choice
of control data (\cite{Thom, Mather}) there is a collection of {\em
  controlled differential forms} (\cite{V, GP}).  A controlled
differential $r$-form (with real coefficients) is a differential
$r$-form $\eta_A$ on each stratum $A$ with the property that, whenever
$A<B$ are strata, then $\eta_B|T_A = \pi_{BA}^*(\eta_A)$ where
$\pi_{BA}:(T_A \cap B) \to A$ is the projection function of a tubular
neighborhood of $A$.  The controlled differential forms constitute a
complex $\Omega^*_{con}(\oX;\RR)$ whose cohomology is naturally
isomorphic to the ordinary cohomology $H^*(\oX;\RR)$.

\begin{lem}  
Fix a Whitney stratification and control data on $\oX$.  Fix $r$ with
$ 0 \le r \le n$ and let $c^r(\Theta)$ be the $r$-th Chern form on
$X$.  For any controlled differential form $\eta \in
\Omega^{2n-r}_{con}(\oX;\CC)$,
\[ 
\int_X c^r(\Theta) \wedge \eta < \infty
\]
\end{lem}
\begin{proof}
Each controlled differential form $\eta$ extends canonically to a
smooth differential form in a neighborhood of $\oX$.  It follows
that if $\widetilde{X} \to \oX$ is a resolution of singularities then
$\eta$ pulls back to a smooth differential form $\tilde\eta$ on
$\widetilde{X}$.  Using the resolution of the Nash blowup as described
in Proposition \ref{prop-main} above, we conclude that
\[ 
\int_{X} c^r(\Theta) \wedge \eta = 
\int_{\widetilde{X}} c^r(\widetilde{\Theta})\wedge \tilde\eta 
< \infty.\qedhere
\]
\end{proof}

\subsection{}
Thus, each Chern form $c^r(\Theta)$ determines a homomorphism
$\Omega^{2n-r}_{con}(\oX) \to \CC$ that vanishes on boundaries, and
hence defines a class in the dual space $H^*(\oX;\CC)^* =
H_*(\oX;\CC)$, which is easily seen to agree with the Mather Chern
class.

\section{Other curvature integrals}\label{sec:MCnos}
\subsection{}
For any partition $I = i_1 + i_2 + \cdots + i_r = n$, the {\em Mather
  Chern number} is the integer $c^{M,I}(\oX)
=\deg \left(c^{i_1}(\xi)\cup \cdots \cup c^{i_r}(\xi) \cap [
  \widehat{X} ]\right)$.  

\begin{prop}\label{prop-integrals}
Let $I = i_1 + \cdots + i_r = n$ be a partition of $n$.  Then
\begin{equation}\label{eqn-partition}
\int_X c^{i_1}(\Theta) \wedge \cdots \wedge c^{i_r}(\Theta) 
= c^{M,I}(\oX).
\end{equation}
that is, the integral of this product of Chern forms equals the
corresponding Mather Chern number.
\end{prop}
The proof is the same as that of Proposition \ref{prop-main}.\qed

Proposition~\ref{prop-integrals} is also proven in \cite{Buryak2}.

One can also define a product of any collection
$c^M_{i_1}(\oX), \dots, c^M_{i_r}(\oX)$ of Mather
Chern classes, as the push-forward
\[
b_*\left(c^{i_1}(\xi)\cup \cdots \cup c^{i_r}(\xi) 
\cap [ \widehat{X} ]\right)
\]
in $H_*(\oX;\CC)$.

\subsection{}
The (algebraic) degree $\deg(\alpha)$ of an integral homology class
$\alpha \in H_{2r}(\PP^N;\ZZ)$ is the unique integer such that $\alpha
= \deg(\alpha)[L^{r}] \in H_{2r}(\PP^N;\ZZ)$ where $L^{r}$ denotes a
codimension $r$ linear subspace of $\PP^N$.

\begin{prop}
Let $\omega$ denote the K\"ahler form on $X$ that is induced from the
ambient projective space.  Then the degree of the Mather Chern class
$c_i^M(\oX)$ is given by the integral,
\begin{equation}\label{eqn-degree}
\deg(c_i^M(\oX)) = \int_X c^{n-i}(\Theta) \wedge \omega^{i}.
\end{equation}
\end{prop}
Equation (\ref{eqn-degree}) follows from the fact that the class
$[\omega]\in H^2(\PP^N)$ represented by the K\"ahler form $\omega$ is
the Poincar\'e dual of the homology class $[L]\in H_2(\PP^N)$
represented by a hyperplane.  \qed

\subsection{}  
The numbers $\deg(c_i^M(\oX))$ may be used to determine the
Gauss Bonnet integrals
\[
\beta_r = \int_{X \cap L^r} \Pf(\Theta_r)
\]
over a generic codimension $r$ linear section $\oX \cap L^r$,
and vice versa.  (Here, $\Theta_r$ is the curvature form on $X \cap
L^r$.) Let
\[ 
\CM(t) = \sum_{r \ge 0} \deg(c^M_r(\oX))t^r
\]
be the Mather Chern (Poincar\'e) polynomial and let
\[ 
\Pf(t) = \sum_{r \ge 0} \beta_r(-t)^r
\]
be the Pfaffian polynomial.  Define an involution on the set of
polynomials of fixed degree, $p \mapsto \mathcal I(p)$ by
\[ 
\mathcal I(p)(t) = \frac{ p(0) + t\, p(-1-t)}{1+t}.
\]

\begin{prop}
The involution $\mathcal I$ interchanges the Mather Chern
polynomial with the Pfaffian polynomial:
\[ 
\CM = \mathcal I(\Pf)\ \text{ and }\ \Pf = \mathcal I(\CM).
\]
\end{prop}
The proof follows the same lines as that of Theorem~1.1 in \cite{Al2};
the key technical step is the good behavior of the Mather Chern class
with respect to generic linear sections, which follows for example
from Lemma~1.2 in~\cite{Parusinski}.\qed

\section{Behavior in families}\label{sec-families}

\subsection{}
We now consider the behavior of the Pfaffian integral along a
smoothing family for~$\oX$.  Localizing the integral in tubular
neighborhoods of components of the singular locus of~$\oX$ in a
smoothing family leads to expressions for invariants of the
singularity in terms of Pfaffian integrals.

As above, $\oX$ denotes a projective variety of dimension~$n$, with
nonsingular part~$X$. We will let $S:=\oX\smallsetminus X$ be the
singular locus of $\oX$.  Suppose that $\oX=\oX_0$ is the central
fiber of a flat family $\ocaX \to D$ over a disk $D\subseteq \CC$
centered at $0$. We let $\oX_\delta=\pi^{-1}(\delta)$ denote the fiber
over $\delta\in D$ and we assume that $\oX_\delta$ is nonsingular for
all $\delta\ne 0$; thus $\ocaX$ is a smoothing of $X$.  We will
further assume that $\ocaX\subseteq \PP^N$ is projective.  It follows from the
first isotopy lemma of R.~Thom (\cite{Mather}) that
$\pi$ is topologically locally trivial over $D\smallsetminus \{0\}$.

A metric is induced on the tangent spaces to $X$ as above, and to
every $\oX_\delta$ for $\delta\ne 0$ in a neighborhood of $0$. We have
an induced connection $\nabla_\delta$ on each $\oX_\delta$, with
curvature $\Theta_\delta\in \Omega^2(\oX_\delta;
\End(T\oX_\delta))$ and Pfaffian 
$\Pf(\Theta_{\RR,\delta})\in \Omega^{2n}(\oX_\delta; \RR)$. 

\begin{prop}\label{prop:EuVer}
Denote by $N_\epsilon(S)$ the $\epsilon$-tubular neighborhood of $S$.
Then 
\[
\lim_{\epsilon\to 0}\, \lim_{\delta\to 0}\, (2\pi)^{-\dim\oX}
\int_{\oX_\delta\cap N_\epsilon(S)} \Pf(\Theta_{\RR,\delta})
=\sum_{Y<X} \left(\sigma_y({\ocaX})-\Eu_y(\oX)\right) \chi(Y)\quad.
\]
\end{prop}

As above, the sum is over the strata $Y$ contained in $S$, and
$\Eu_y(\oX)$ denotes the (common) value of the local Euler obstruction
at points $y$ of the stratum $Y$. Similarly, $\sigma_y(\ocaX)$ denotes
the value at $y\in Y$ of Verdier's {\em specialization function\/}
$\sigma(\ocaX)=\sigma_*(\one_{\ocaX})$, see~\cite[\S3]{Verdier}.  (We
assume that $\sigma(\ocaX)$ is locally constant along strata.)
Explicitly,
\[
\sigma_y(\ocaX)=\lim_{\epsilon\to 0} \lim_{\delta\to 0} \chi(\oX_\delta\cap B_\epsilon(y))
\]
where $B_\epsilon(y)$ is the $\epsilon$-ball centered at $y$.  Note
that both $\Eu(\oX)$ and $\sigma(\ocaX)$ equal $1$ away from~$S$.
Therefore their difference is supported on $S$.

If $\ocaX\subseteq \overline{\mathcal M}$ may be realized as a
hypersurface in a flat family $\overline{\mathcal M} \to D$ with
nonsingular fibers, then every fiber $\oX_\delta$, $\delta\in D$ is
itself a hypersurface of a nonsingular variety. If $\oX=\oX_0$ has
isolated singularities, then Proposition~\ref{prop:EuVer} admits the
following particularly explicit formulation,
which recovers a formula of R.~Langevin (\cite[Theorem 1]{Langevin}):

\begin{cor}\label{cor:Milnor}
In the situation considered above, assume that $\oX$, $\oX_\delta$ are
hypersurfaces of nonsingular varieties and that $\dim S=0$. Then
\[
\lim_{\epsilon\to 0}\, \lim_{\delta\to 0}\, (2\pi)^{-\dim\oX}
\int_{\oX_\delta\cap N_\epsilon(S)} \Pf(\Theta_{\RR,\delta})
=(-1)^{\dim \oX}\sum_{p\in S}\big(\mu_\oX(p)+\mu_{\oX\cap H}(p)\big)\quad,
\]
where $\mu(p)$ denotes the Milnor number at $p$ and $H$ is a general
hyperplane through $p$.
\end{cor}

In particular, the left-hand-side is independent of the smoothing.

\subsection{}
We prove Proposition~\ref{prop:EuVer} and Corollary~\ref{cor:Milnor}.

\begin{proof}[Proof of Proposition~\ref{prop:EuVer}.]
Consider the complements $\oX_\delta\smallsetminus N_\epsilon(S)$ of
the tubular neighborhood. For all positive $\epsilon$ near $0$ the
integral
\[
\int_{\oX_\delta\smallsetminus N_\epsilon(S)} \Pf(\Theta_{\RR,\delta})
\]
is a continuous function of $\delta$ near $0$, therefore
\[
\lim_{\epsilon\to 0}\, \lim_{\delta\to 0}\, (2\pi)^{-\dim\oX}
\int_{\oX_\delta\smallsetminus N_\epsilon(S)} \Pf(\Theta_{\RR,\delta})
=\lim_{\epsilon\to 0} \,
(2\pi)^{-\dim \oX}\int_{\oX\smallsetminus N_\epsilon(S)} \Pf(\Theta_{\RR, 0})
=\deg c_0^M(\oX)
\]
by Proposition~\ref{prop-main}.  On the other hand, let
$\chi=\chi(X_\delta)$ for any $\delta$ near $0$.  (By hypothesis, this
number does not depend on $\delta$.) Then
\[
\lim_{\delta\to 0}\, (2\pi)^{-\dim \oX}
\int_{\oX_\delta} \Pf(\Theta_{\RR,\delta})
=\chi
\]
by the ordinary Gauss Bonnet formula. Since $\oX_\delta\cap
N_\epsilon(S) =\oX_\delta \smallsetminus
\left(\oX_\delta\smallsetminus N_\epsilon(S)\right)$, this gives
\[
\lim_{\epsilon\to 0}\, \lim_{\delta\to 0}\, (2\pi)^{-\dim \oX}
\int_{\oX_\delta\cap N_\epsilon(S)} \Pf(\Theta_{\RR,\delta})
=\chi-\deg c_0^M(X)
\]
Arguing as in~\S\ref{ss:Eu},
\begin{equation}\label{eq:c0M}
\deg c_0^M(\oX) =\sum_Y \Eu_y(\oX) \chi(Y)\quad.
\end{equation}
By Verdier's specialization theorem \cite[Th\'eor\`eme 5.1]{Verdier},
we have, for $\delta$ near $0$:
\begin{equation}\label{eq:sigma}
\chi = \deg \sigma_* c_*(\one_{\oX_\delta})
=\deg c_* \sigma_* (\one_{\ocaX})
=\deg c_*(\sum_Y \sigma_y(\ocaX) \one_Y) 
=\sum_Y \sigma_y(\ocaX) \chi(Y)\quad.
\end{equation}
Identities~\eqref{eq:c0M} and~\eqref{eq:sigma} complete the proof of
Proposition~\ref{prop:EuVer}.
\end{proof}

\begin{proof}[Proof of Corollary~\ref{cor:Milnor}]
For an isolated hypersurface singularity, 
\[
\sigma_\ocaX(p) = 1 + (-1)^{\dim\oX} \mu_{\oX} (p)
\]
as a consequence of~\cite[Proposition 5.1]{ParuPra}. On the other
hand,
\[
\Eu_\oX(p) = 1 - (-1)^{\dim\oX} \mu_{\oX\cap H}(p)
\]
as proven in~\cite{Piene} (also cf.~\cite[Remarque, p.~240]{Dubson}).
The formula in Corollary~\ref{cor:Milnor} follows then immediately from 
Proposition~\ref{prop:EuVer}.
\end{proof}

\subsection{}
Assume now that $\oX$ is a hypersurface in a nonsingular projective
variety, with a single singular point $p$, and that $\ocaX \to D$ is a
smoothing family as above.
If $H$ is a general hyperplane through $p$, then $\ocaX\cap H$ is a
smoothing family for $\oX\cap H$ over (a possibly smaller disk) $D$,
satisfying the same hypotheses as $\ocaX$.  (In fact the set of
parameter values $z\in D\smallsetminus \{0\}$ such that $H$ fails
to be transversal to the fiber $\overline{X}_z = \pi^{-1}(z)$ is
algebraic, hence finite.  By shrinking $D$ if necessary we may 
assume that it contains no such points.)
Iterating, we can find general hyperplanes $H_1,\dots, H_{\dim\oX-1}$
through $p$ such that
\[
\ocaX^{(r)}:=\ocaX\cap H_1\cap\cdots\cap H_r
\]
satisfies the same hypotheses for $r=1,\dots, \dim\oX-1$ (perhaps at
the price of further reducing the radius of the disk $D$).

Assume this is the case and let $\ocaX^{(0)}=\ocaX$. We have fibers
$\oX_\delta^{(r)}$ and corresponding Pfaffians
$\Pf(\Theta_{\RR,\delta}^{(r)})$ for $r=0,\dots, \dim\oX-1$.  We let
$H$ be one further general hyperplane; the intersection
\[
\oX^{(\dim\oX)}=\oX\cap H_1\cap\cdots\cap H_{\dim\oX-1}\cap H
\]
is then zero-dimensional. It consists of an $m$-tuple point at $p$,
where $m$ is the multiplicity of $\oX$ at $p$, and of a set of reduced
points.

\begin{cor}\label{cor:oneMilnor}
Let $\oX$ be a hypersurface in a projective nonsingular variety;
assume $\oX$ has an isolated singularity $p$, of multiplicity $m$ and
Milnor number $\mu_\oX(p)$.  Then with notation as above:
\begin{equation}\label{eq:milnor}
\lim_{\epsilon\to 0}\, \lim_{\delta\to 0}\, \sum_{i=1}^{\dim \oX} (2\pi)^{-\dim\oX+i-1}
\int_{\oX_\delta^{(i-1)}\cap N_\epsilon(S)} \Pf(\Theta_{\RR,\delta}^{(i-1)})
=(-1)^{\dim \oX}\mu_{\oX}(p)-m+1\quad.
\end{equation}
\end{cor}

\begin{proof}
We can evaluate each summand by using
Corollary~\ref{cor:Milnor}. This shows that the left-hand side equals
\begin{multline*}
(-1)^{\dim \oX} (\mu_{\oX}(p) + \mu_{\oX^{(1)}}(p))
+(-1)^{\dim \oX-1} (\mu_{\oX^{(1)}}(p) + \mu_{\oX^{(2)}}(p)) \\
+\cdots 
+(-1)^{1} (\mu_{\oX^{(\dim \oX-1)}}(p) + \mu_{\oX^{(\dim \oX)}}(p))
\end{multline*}
that is (`telescoping') $(-1)^{\dim \oX} \mu_{\oX}(p) - \mu_{\oX^{(\dim \oX)}}(p)
=(-1)^{\dim \oX} \mu_{\oX}(p)-m+1$, as stated.
\end{proof}

An integral formula for the Milnor number was obtained by Phillip Griffiths 
(\cite{Griffiths}, \cite{Kennedy}; and see~\cite[p.~207]{Loeser}). 
It would be interesting to establish a direct relation between the two formulas.

\begin{example}
The simplest example illustrating the situation considered here is
probably a family of nonsingular conics degenerating to the union
$\overline X$ of two lines, $xy = \delta z^2$ in
$\PP^2$. Topologically, smooth conics and complex projective lines are
$2$-spheres; as $\delta\to 0$, the degeneration may be pictured as
follows.
\begin{center}
\includegraphics[scale=.25]{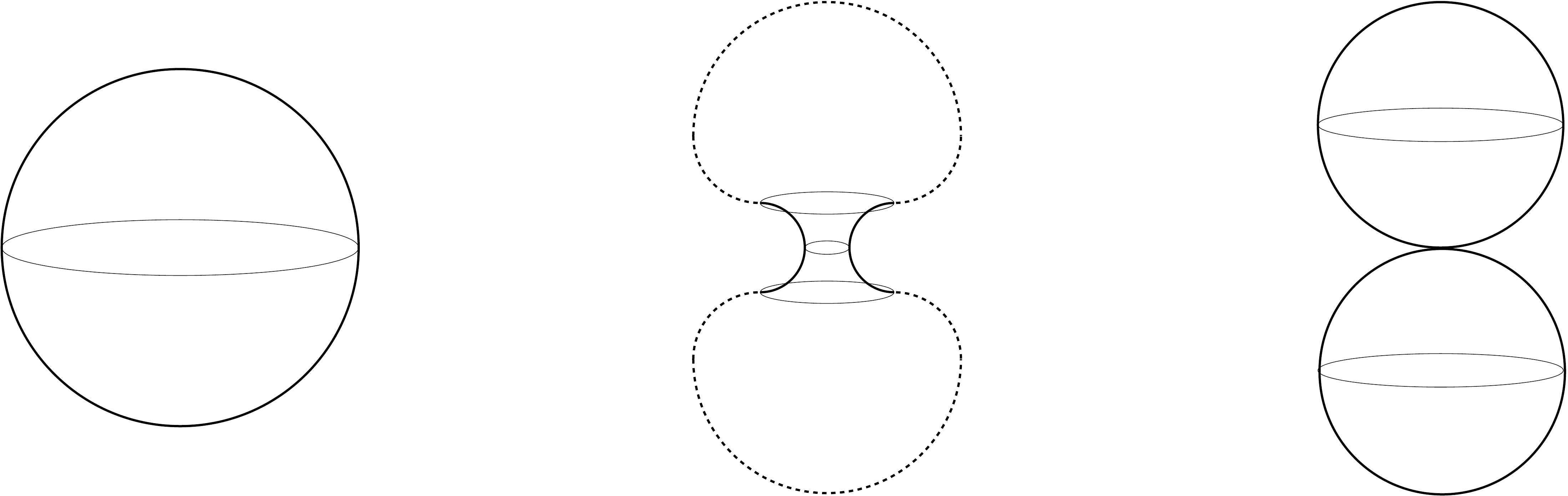}
\end{center}
where the saddle in the middle depicts the intersection with an
$\epsilon$-polydisk centered at $p$. In this example, the integral in
the left-hand side of~\eqref{eq:milnor} may be evaluated explicitly;
as $\delta\to 0$, it converges to
\[
\frac{2(\epsilon^2-1)}{(\epsilon^2+1)}
\]
and hence to $-2$ as $\epsilon \to 0$. According to
Corollary~\ref{cor:oneMilnor},
\[
-2 = (-1)^1 \mu_{\oX}(p)-2+1\quad,
\]
i.e., $\mu_{\oX}(p)=1$ (as it should).
\end{example}

\section{Extensions and further questions}\label{sec:Efq}
\subsection{} 
For simplicity of exposition we assumed that $\oX$ is projective, but
Theorem \ref{thm-integral} applies to subvarieties of any algebraic
K\"ahler manifold.

\subsection{}
In fact, an analogue of Proposition~\ref{prop-main} holds (with the
same proof) for the Mather Chern class of any coherent sheaf
(\cite{Schwartz}, \cite{Kwie}) realized as a quotient of a locally
free sheaf~$\mathcal E$, with respect to any hermitian metric defined
on the vector bundle of sections of the dual~$\mathcal E^\vee$.

\subsection{}
Homology classes cannot, in general, be multiplied
together. Therefore, multiplicativity properties of Mather Chern
classes are subtle.  It follows from \cite{Barthel} and \cite{BW} that
these classes lift canonically to middle intersection homology, giving
a well defined homology-valued product of any two such classes.
Proposition \ref{prop-integrals} provides an invariant interpretation
of any top-degree product of Mather Chern classes and, as observed
there, a product of any collection of Mather-Chern classes may be
defined by means of a corresponding product in the Nash blow-up.
We do not know whether the vector space spanned by these products of
Mather Chern classes forms a natural well-defined ring within the
homology $H_*(\oX)$.

\subsection{} 
Pontrjagin classes of compact complex manifolds also admit Chern-Weil
descriptions, and the signature of such a manifold may be expressed as
a (rational) linear combination of Chern numbers, hence as a curvature
integral.  If $\oX$ is a singular variety, the same linear combination
of Mather-Chern numbers is defined and by Proposition
\ref{prop-integrals} it is given by the same curvature integral over
the nonsingular part, $X$.  Is this number an integer?  What is the
relation between this number and the (intersection homological)
signature of $\oX$?

\end{document}